\documentclass[12pt, a4paper]{amsart}
\usepackage{amsmath,amsthm,amssymb,amscd}
\usepackage[shortalphabetic]{amsrefs}
\usepackage[all]{xy}
\usepackage{empheq}
\usepackage{color}
\usepackage{hyperref}
\usepackage{graphicx}
\usepackage{epstopdf}
\usepackage{epsfig}

\newtheorem{thm}{Theorem}[section]

\newtheorem{cor}[thm]{Corollary}
\newtheorem{lem}[thm]{Lemma}
\newtheorem{pro}[thm]{Proposition}
\theoremstyle{definition}
\newtheorem{eg}[thm]{Example}
\newtheorem{de}[thm]{Definition}

\newtheorem*{re}{Remark}

\newcommand{\excise}[1]{}

\DeclareMathOperator{\diam}{diam}
\DeclareMathOperator{\post}{post}

\def\C{\mathbb{C}}

\def\Z{\mathbb{Z}}
\def\S{\mathbb{S}}

\def\N{\mathbb{N}}

\def\G{\mathcal{G}}

\def\T{\mathcal{T}}

\def\D{\mathcal{D}}

\def\:{\colon}

\def\ra{\rightarrow}

\title{Thurston maps and asymptotic upper curvature}
\author{Qian Yin}
\thanks{The author was partially supported by NSF grants  DMS 0757732, DMS 0353549, DMS 0456940, DMS 0652915, DMS 1058772, and DMS 1058283.}
\address{Qian Yin \\The University of Chicago\\ 5734 S. University Avenue\\ Chicago, Illinois 60637\\USA}
\email{qyin@math.uchicago.edu}

\begin{document}

\begin{abstract}
A Thurston map is a branched covering map from $\S^2$ to $\S^2$ with a finite postcritical set. We associate a natural Gromov hyperbolic graph $\G=\G(f,\mathcal C)$ with an expanding Thurston map $f$ and a Jordan curve $\mathcal C$ on $\S^2$ containing $\post(f)$. The boundary at infinity of $\G$ with associated visual metrics can be identified with $\S^2$ equipped with the visual metric induced by the expanding Thurston map $f$. We define asymptotic upper curvature of an expanding Thurston map $f$ to be the asymptotic upper curvature of the associated Gromov hyperbolic graph, and establish a connection between the asymptotic upper curvature of $f$ and the entropy of $f$.
\end{abstract}

\maketitle

\tableofcontents

\section{Introduction}
\noindent
A \emph{rational map} $f\:\widehat\C \ra \widehat\C$ is a map on the Riemann sphere $\widehat\C=\C \cup \{\infty\}$ which can be written as a quotient of two relatively prime complex polynomials $p(z)$ and $q(z)$, with $q(z)\not=0$,
\begin{eqnarray}\label{rationalmap}
f(z)=\frac{p(z)}{q(z)}=
\frac{a_0z^{m}+\ldots+a_{m}}{b_0z^l+\ldots+b_l},
\end{eqnarray}
where $a_i,b_j \in \C$ for $i=0,\ldots,m$ and $j=0,\ldots, l$.
The \emph{postcritical set} $\post(f)$ of $f$ is defined to be the forward orbits of the critical points
\begin{eqnarray*}
\post(f)=\bigcup_{n\geq 1}\{f^n(c)\: c\in {\rm crit}(f)\}.
\end{eqnarray*}
If the postcritical set $\post(f)$ is finite, we say that the map $f$ is \emph{postcritically finite}.

Thurston introduced a topological analog of a postcritically finite rational map, now known as a \emph{Thurston map} (see \cite{DHThurston}). A \emph{Thurston map} $f\:\S^2\ra \S^2$ is a branched covering map with finite postcritical set $\post(f)$.
The notion of an expanding Thurston map was introduced in \cite{BMExpanding} as a topological analog of a postcritically finite rational map whose Julia set is the whole sphere $\widehat{\C}$. Roughly speaking, a Thurston map is called \emph{expanding} if all the connected components of the preimage under $f^{-n}$ of any open Jordan region disjoint from $\post(f)$ become uniformly small as $n$ tends to infinity. We refer the reader to Definition~\ref{expandingmap} for a more precise statement. A related and more general notion of expanding Thurston maps was introduced in \cite{HPCoarse}. Latt\`es maps are among the simplest examples of expanding Thurston maps.


Let $f$ be an expanding Thurston map, and let $\mathcal C$ be a Jordan curve containing $\post(f)$.
The Jordan Curve Theorem implies that $\S^2\setminus\mathcal C$ has precisely two connected components, whose closures we call \emph{$0$-tiles}. We call the closure of each connected component of the preimage of $\S^2\setminus\mathcal C$ under $f^n$ an \emph{$n$-tile}. In Section 5 of \cite{BMExpanding}, it is proved that the collection of all $n$-tiles gives a cell decomposition of $\S^2$.

Every expanding Thurston map $f\:\S^2\ra \S^2$ induces a natural class of metrics on $\S^2$, called \emph{visual metrics} (see Definition \ref{visual}), and each visual metric $d$ has an associated \emph{expansion factor} $\Lambda > 1$. This visual metric is essentially characterized by the geometric property that the diameter of an $n$-tile is about $\Lambda^{-n}$, and the distance between two disjoint $n$-tiles is at least about $\Lambda^{-n}$. The supremum of the expansion factors of all visual metrics is called the \emph{combinatorial expansion factor} $\Lambda_0$ (see \cite[Theorem 1.5]{BMExpanding}). For Latt\`es maps, the supremum is obtained. In general, the supremum is not obtained.

\bigskip

A geodesic metric space $(X,d)$ is called a \emph{Gromov hyperbolic} space if every geodesic triangle in it is ``very thin''. It can also defined in terms of \emph{Gromov products}.
For any points $x,y,p\in X$, the \emph{Gromov product} $(x,y)_p$ of $x$ and $y$ with respect to the base point $p$ is defined as
\begin{eqnarray} \label{gproduct}
  (x,y)_p \:= \frac12 \left[d(x,p)+d(y,p)-d(x,y) \right].
\end{eqnarray}
The space $X$ is called \emph{$\delta$-hyperbolic} (or \emph{Gromov hyperbolic}) for some $\delta\geq 0$ if there exists a base point $p\in X$ such that for all $x,y,z\in X$, we have
\begin{eqnarray} \label{trianglein}
  (x,y)_p\geq \min\{(x,z)_p,(z,y)_p\}-\delta.
\end{eqnarray}

We construct a graph $\G=\G(f,\mathcal C)$ by letting the tiles in the cell decompositions of $(f, \mathcal C)$ be vertices of $\G$.
There is an edge between the two vertices $X^n,Y^m\in V$, denoted
$X^n\sim Y^m$ if as underlying tiles
\[|n-m|\leq 1 \mbox{ and } X^n\cap Y^m\not= \emptyset.\]
It turns out that the graph $\G$ with the path metric is a Gromov hyperbolic space (see Theorem \ref{gh}).
\begin{thm} 
Let $f\: \S^2 \ra \S^2$ be an expanding Thurston map 
and let $\mathcal C \subset \S^2$ be a Jordan curve containing $\post(f)$. Then the graph $\G(f,\mathcal C)$ equipped with the path metric $\eta$ is a Gromov hyperbolic space.
\end{thm}

There is a natural boundary at infinity of a Gromov hyperbolic space. Roughly speaking, the \emph{boundary at infinity} is the set of equivalence classes of geodesic rays in the Gromov hyperbolic space. It can also be equipped with a \emph{Gromov product} by taking infimum of the infimum limit of the Gromov product along all the geodesic rays among the corresponding equivalence classes.
 A \emph{visual metric} $\rho$ on the boundary at infinity of a Gromov hyperbolic space is a metric that has a bounded ratio
\[ \rho(\xi,\xi')/ \Lambda^{-(\xi,\xi')_p}\]
for some fixed $\Lambda>1$ and  for all points $\xi$ and $\xi'$ on the boundary.


In Proposition \ref{samevisual}, we show the following:
\begin{pro} 
For an expanding Thurston map $f$ and a Jordan curve $\mathcal C\subset \S^2$ containing $\post (f)$, the boundary at infinity $\partial_{\infty}\G$ of the graph tile $\G(f,\mathcal C)$ can be identified with $\S^2$.
Under this identification, a metric $d$ is a visual metric on $\S^2$ with respect
to the expanding Thurston map $f$ if and only if $d$ is a visual metric on $\partial_{\infty}\G$ (in the sense of Gromov hyperbolic spaces).
\end{pro}

We deduce that for any Jordan curves $\mathcal C$ and $\mathcal C'$ containing $\post (f)$, the classes of visual metrics on $\partial_{\infty}\G(f,\mathcal C)$ and $\partial_{\infty}\G(f,\mathcal C')$  can also be identified (see Corollary \ref{corsamevisual}). A similar graph to $\G(f,\mathcal C)$ has also been studied by Kevin Pilgrim in \cite{PilJulia}, from a somewhat different point of view. Our results overlap in some special cases. He considers the map $f$ being $C^1$ and $\S^2\setminus \post{(f)}$ equipped with a special Riemannian metric, and prove that the Julia set of $f$ can be identified as the Gromov boundary of a certain Gromov hyperbolic one-complex.

In \cite{BFAsymptotic}, the \emph{asymptotic upper curvature} of a Gromov hyperbolic space is introduced. It is the analog of sectional curvature on Riemannian manifolds. Fix $\kappa\in [-\infty,0)$. We call a metric space $X$ an \emph{AC$_u(\kappa)$-space} if there exists $p\in X$ and a constant $c\geq 0$ such that for all $x,x'\in X$ and all finite sequences $x_0=x,x_1,\ldots,x_n=x'$ in $X$,
\begin{eqnarray}\label{acspaceeq}
(x,x')_p\geq \min_{i=1,\ldots,n}(x_{i-1},x_i)_p-\frac1{\sqrt{-\kappa}}\log n-c.
\end{eqnarray}
Here we use the convention $1/\sqrt{\infty}=0$.
We call
\begin{eqnarray*}
K_u(X):=\inf\{\kappa\ \in [-\infty,0): X \mbox{ is an AC$_u(\kappa)$-space} \}
\end{eqnarray*}
the \emph{asymptotic upper curvature} of $X$. It is invariant under rough-isometry.

For any Jordan curves $\mathcal C$ and $\mathcal C'$ containing $\post (f)$, the metric spaces $\G=\G(f,\mathcal C)$ and $\G'=\G(f,\mathcal C')$ are rough-isometric (see Proposition~\ref{roughgg}). Hence we may define the \emph{asymptotic upper curvature} $K_u(f)$ of an expanding Thurston map $f$ as 
\begin{eqnarray} \label{asyf}
K_u(f):= K_u(\G(f,\mathcal C)),
\end{eqnarray}
where $\mathcal C \subseteq \S^2$ is any Jordan curve containing $\post(f)$. Using the notation above, we have the following theorem (see Theorem~\ref{main2}).

\begin{thm} \label{main3}
Let $f \: \S^2 \ra \S^2$ be an expanding Thurston map. The asymptotic upper curvature of $f$ satisfies
\[K_u(f)\geq-\frac14\log^2(\deg f). \]
If in addition, the map $f$ has no periodic critical points, then the tile graph $\G=\G(f)$ is an AC$_u(\kappa)$-space with
\[ \kappa= -\frac14\log^2(\deg f),\]
if and only if the map $f$ is topologically conjugate to a Latt\`es map.
\end{thm}

Finally, we explain how Theorem~\ref{main3} may be interpreted as adding a new and important piece to the Sullivan dictionary.
The Sullivan dictionary is a collection of correspondences between concepts and results from the field of Kleinian groups and
the field of iterated maps on $\S^2$.

The dictionary first appeared alongside the proof of the \emph{no wandering domain theorem} \cite{SulQuasiconformal1}.
Recall that the \emph{Fatou set} of a rational map $f$ on the Riemann sphere
is the set of points whose nearby points
stay close together under iteration of $f$.
The no wandering domain theorem states that each connected component of the Fatou set is eventually periodic under iteration by $f$.
Sullivan proved this result, observing that it was analogous to \emph{Ahlfors' finiteness theorem} \cite{AhlRemarks}. Recall that a Kleinian group $\Gamma$ acts on the Riemann sphere via M\"obius transformations, and that its
\emph{domain of discontinuity} is the largest open subset of the Riemann sphere on which $\Gamma$ acts properly discontinuously.
If $\Gamma$ is finitely generated, then Ahlfors' finiteness theorem implies that every connected component of the domain of discontinuity has a non-trivial stabilizer, i.e. there are no wandering domains.

Sullivan argued that the Fatou set of a rational map plays an analogous role to the domain of discontinuity of a Kleinian group, along with many other such correspondences (see Table~\ref{t:dict} for a sample, as well as \cite[Section~2]{RRKleinian}). Since its original appearance in \cite{SulQuasiconformal1}, the Sullivan dictionary has been significantly developed, and has continued to provide inspiration for research across both fields. For example, the recently proved \emph{density theorem for Kleinian groups} \cite{NSNonrealizability,OhsRealising} provides a counterpart to the classical \emph{density theorem for structurally stable rational maps} \cite{MRSDynamics}.



\begin{table}\footnotesize
    \begin{tabular}{| c | c |}
  \hline
  {\bf Dynamical Systems} & {\bf Kleinian Group }\\ \hline
  Fatou sets & Domains of discontinuity \\ \hline
  Julia sets & Limit sets\\ \hline
  Periodic points  & Fixed points \\ \hline
  Repelling periodic points are dense  & Loxodromic fixed points are dense \\
  in the Julia set &in the limit set \\ \hline
  Mandelbrot set & Parameter space \\ \hline
  Sullivan's no wandering domain theorem & Ahlfors' finiteness theorem\\ \hline
    Shishikura's sharp bound on number  & Ber's area theorem\\
  of Fatou cycles& \\ \hline
  Density theorem of structural stable maps & Density theorem for Kleinian groups\\ \hline
  No invariant line fields on Julia set  & No invariant line fields on limit set \\
  (conjecture)& \\ \hline
  Theorem \ref{main3} & Hamenst{\"a}dt's entropy rigidity theorem \\ \hline
    \end{tabular}
    \vspace{.3cm}
  \caption{Sullivan dictionary}\label{t:dict}
\end{table}


One famous result from the field of Kleinian groups with no known analogue under the Sullivan dictionary is
 \emph{Hamenst{\"a}dt's entropy rigidity theorem} \cite{HamEntropy}, which establishes a connection between the curvature of a compact manifold $M$ and the topological entropy of the geodesic flow on the tangent bundle of $M$. More precisely, let $h(N)$ denote the topological entropy of the geodesic flow on the unit tangent bundle of a Riemannian manifold $N$. The theorem states that if $M$ is homotopic to a compact quotient $S$ of a hyperbolic space and
 $h(M) \leq h(S)$, then the maximum of the sectional curvature of $M$ is bounded below by $-1$, and is equal to $-1$ if and only if
$M$ and $S$ are isometric.

Theorem \ref{main3} provides the counterpart to Hamenst{\"a}dt's entropy rigidity theorem under the Sullivan dictionary. Indeed, Corollary~20.8 in \cite{BMExpanding} shows that the topological entropy $h(f)$ of an expanding Thurston map $f$ is $\log(\deg(f))$. Hence, the theorem provides a lower bound on the asymptotic upper curvature of an expanding Thurston map $f$ in terms of the topological entropy of $f$, i.e.
\[K_u(f)\geq-\frac14h(f)^2, \]
together with a condition for equality to hold.



\bigskip
\noindent
\textbf{Acknowledgements.} This paper is part of the author's PhD thesis under the supervision of Mario Bonk. The author would like to thank Mario Bonk for introducing her to and teaching her about the subject of Thurston maps and its related fields. The author is inspired by his enthusiasm and mathematical wisdom, and is especially grateful for his patience and encouragement.
The author would like to thank Dennis Sullivan for valuable conversations and sharing his mathematical insights. 
The author benefited greatly from Dick Canary's mini-course on the Kleinian group aspects of the Sullivan dictionary.
The author also would like to thank Michael Zieve and Alan Stapledon for useful comments and feedback.

\section{Expanding Thurston maps and Cell Decompositions} \label{expanding}
\noindent
In this section we review some definitions and facts on expanding Thurston maps. We refer the reader to Section 3 in \cite{BMExpanding} for more details. We write $\N$ for the set of positive integers, and $\N_0$ for the set of non-negative integers. We denote the identity map on $\S^2$ by ${\rm id}_{\S^2}$.

Let $\S^2$ be a topological 2-sphere with a fixed orientation. A continuous map $f\:\S^2\ra \S^2$ is called \emph{a branched covering map} over $\S^2$ if $f$ can be locally written as
\[z\mapsto z^d\]
under certain orientation-preserving coordinate changes of the domain and range. More precisely, we require that for any point $p\in \S^2$, there exists some integer $d>0$, an open neighborhood $U_p\subseteq \S^2$ of $p$, an open neighborhood $V_q\subseteq \S^2$ of $q=f(p)$, and orientation-preserving homeomorphism
\[\phi\: U_p\ra U\subseteq \C\]
and
\[\psi \: V_p\ra V\subseteq \C\]
with $\phi(p)=0$ and $\psi(q)=0$ such that
\[(\psi\circ f \circ \phi^{-1} )(z)=z^d\]
for all $z\in U$. The positive integer $d=\deg_f(p)$ is called the \emph{local degree} of $f$ at $p$ and only depends on $f$ and $p$. A point $p\in \S^2$ is called a \emph{critical point} of $f$ if $\deg_f(p)\geq 2$, and a point $q$ is called \emph{critical value} of $f$ if there is a critical point in its preimage $f^{-1}(q)$. If $f$ is a branched covering map of $\S^2$, $f$ is open and surjective. There are only finitely many critical points of $f$ and $f$ is \emph{finite-to-one} due to the compactness of $\S^2$. Hence, $f$ is a covering map away from the critical points in the domain and critical values in the range. The \emph{degree $\deg(f)$} of $f$ is the cardinality of the preimage over a non-critical value. In addition, we have
\[\deg(f)=\sum_{p\in f^{-1}(q)}\deg_f(p)\]for every $q\in \S^2$.

For $n\in \N$, we denote the $n$-th iterate of $f$ as
\[f^n=\underbrace{f\circ f\circ \cdots \circ f}_{\textstyle{n} \mbox{ factors}}.\]
We also set $f^0={\rm id}_{\S^2}$.

If $f$ is a branched cover of $\S^2$, so is $f^n$, and
\[\deg(f^n)=\deg(f)^n.\]Let crit$(f)$ be the set of all the critical points of $f$. We define the set of \emph{postcritical points} of $f$ as
\[\post(f)=\bigcup_{n\in \N}\{f^n(c)\: c\in {\rm crit}(f)\}.\]We call a map $f$ \emph{postcritically-finite} if the cardinality of $\post(f)$ is finite. Notice that $f$ is postcritically-finite if and only if there is some $n\in \N$ for which $f^n$ is postcritically-finite.

Let $\mathcal{C}\subseteq \S^2$ be a Jordan curve containing $\post(f)$. We fix a metric $d$ on $\S^2$ that induces the standard metric topology on $\S^2$.
Denote by \emph{${\rm mesh}(f,n,{\mathcal C})$} the supremum of the diameters of all connected components of the set $f^{-n}(\S^2\setminus {\mathcal C})$.

\begin{de} \label{expandingmap}
A branched covering map $f\:\S^2\ra \S^2$ is called a \emph{Thurston map} if $\deg(f)\geq 2$ and $f$ is postcritically-finite. A Thurston map $f\:\S^2\ra \S^2$ is called \emph{expanding} if there exists a Jordan curve $\mathcal{C}\subseteq \S^2$ with $\mathcal{C} \supseteq \post(f)$ and
\begin{equation} \label{mesh}
\lim_{n\ra \infty}{\rm mesh}(f,n,{\mathcal C})=0.
\end{equation}
\end{de}

The relation \eqref{mesh} is a topological property, as it is independent of the choice of the metric, as long as the metric induces the standard topology on $\S^2$. Lemma 8.1 in \cite{BMExpanding} shows that if the relation \eqref{mesh} is satisfied for one Jordan curve ${\mathcal C}$ containing $\post(f)$, then it holds for every such curve. One can essentially show that a Thurston map is expanding if and only if all the connected components in the preimage under $f^{-n}$ of any open Jordan region not containing $\post(f)$  become uniformly small as $n$ goes to infinity.

The following theorem (Theorem 1.2 in \cite{BMExpanding}) says that there exists an invariant Jordan curve for some iterates of $f$.
\begin{thm} \label{invariantJordancurve}
If $f \: \S^2\ra \S^2$ is an expanding Thurston map, then for some $n\in \N$ there exists a Jordan curve $\mathcal{C}\subseteq \S^2$ containing $\post(f)$ such that $\mathcal C$ is invariant under $f^n$, i.e., $f^n({\mathcal C})\subseteq {\mathcal C}$.
\end{thm}

Recall that an \emph{isotopy} $H$ between two homeomorphisms is a homotopy so that at each time $t\in [0,1]$, the map $H_t$ is a homeomorphism. An \emph{isotopy $H$ relative to a set $A$} is an isotopy satisfying
\[H_t(a)=H_0(a)=H_1(a)\]
for all $a\in A$ and $t\in [0,1]$.

\begin{de}
Consider two Thurston maps $f\:\S^2\ra \S^2$ and $g\:\S^2_1\ra \S^2_1$, where $\S^2$ and $\S^2_1$ are $2$-spheres. We call the maps $f$ and $g$ \emph{(Thurston) equivalent} if there exist homeomorphisms $h_0,h_1\:\S^2\ra \S^2_1$ that are isotopic relative to $\post(f)$ such that $h_0\circ f=g\circ h_1$.
We call the maps $f$ and $g$ \emph{topologically conjugate} if there exists a homeomorphism $h\:\S^2\ra \S^2_1$ such that $h\circ f=g\circ h$.
\end{de}
For equivalent Thurston maps, we have the following commutative diagram
\[\begin{CD}
\S^2 @>h_1>>\S^2_1 \\
@Vf VV @VVg V\\
\S^2 @>h_0>> \S^2_1  .
\end{CD} \]



We now consider the cardinality of the postcritical set of $f$. In Remark 5.5 in \cite{BMExpanding}, it is proved that there are no Thurston maps with $\#\post(f)\leq 1$. Proposition 6.2 in \cite{BMExpanding} shows that all Thurston maps with $\#\post(f)=2$ are Thurston equivalent to a \emph{power map}  on the Riemann sphere,
\[z\mapsto z^k, \mbox{ for some }k\in \Z\setminus\{-1,0,1\}.\]
Corollary 6.3 in \cite{BMExpanding} states that if $f\:\S^2\ra\S^2$ is an expanding Thurston map, then $\#\post(f)\geq 3$.

Let $f\:\S^2\ra\S^2$ be a Thurston map, and let ${\mathcal C}\subseteq \S^2$ be a Jordan curve containing $\post(f)$. By the Sch\"onflies theorem, the set $\S^2\setminus {\mathcal C}$ has two connected components, which are both homeomorphic to the open unit disk. Let $T_0$ and $T_0'$ denote the closures of these components. They are cells of dimension $2$, which we call \emph{$0$-tiles}. The postcritical points of $f$ are called \emph{$0$-vertices} of $T_0$ and $T_0'$, singletons of which are cells of dimension $0$. We call the closed arcs between vertices \emph{$0$-edges} of $T_0$ and $T_0'$, which are cells of dimension $1$. These $0$-vertices, $0$-edges and $0$-tiles form a cell decomposition of $\S^2$, denoted by $\D^0=\D^0(f,{\mathcal C})$. We call the elements in $\D^0$ $0$-cells. Let $\D^1=\D^1(f,{\mathcal C})$ be the set of connected subsets $c\subseteq \S^2$ such that $f(c)$ is a cell in $\D^0$ and $f|_c$ is a homeomorphism of $c$ onto $f(c)$. Call $c$ a $1$-tile if $f(c)$ is a $0$-tile, call $c$ a $1$-edge if $f(c)$ is a $0$-edge, and call $c$ a $1$-vertex if $f(c)$ is a $1$-vertex. Lemma 5.4 in \cite{BMExpanding} states that $\D^1$ is a cell decomposition of $\S^2$. Continuing in this manner, let $\D^n=\D^n(f,{\mathcal C})$  be the set of all connected subsets of $c\subseteq \S^2$ such that $f(c)$ is a cell in $\D^{n-1}$ and $f|_c$ is a homeomorphism of $c$ onto $f(c)$, and call these connected subsets $n$-tiles, $n$-edges and $n$-vertices correspondingly, for $n\in\N_0$. By Lemma 5.4 in \cite{BMExpanding}, $\D^n$ is a cell decomposition of $\S^2$, for each $n\in \N_0$, and we call the elements in $\D^n$ $n$-cells. The following lemma lists some properties of these cell decompositions. For more details, we refer the reader to Proposition 6.1 in \cite{BMExpanding}.
\begin{lem}\label{tilenumber}
Let $k,n\in \N_0$, let $f\:\S^2\ra \S^2$ be a Thurston map, let $\mathcal C\subset~ \S^2$ be a Jordan curve with $\mathcal C\supseteq \post(f)$, and let $m=\#\post(f)$.
\begin{enumerate}
  \item 
     If $\tau$ is any $(n+k)$-cell, then $f^k(\tau)$ is an $n$-cell, and $f^k|_{\tau}$ is a homeomorphism of $\tau$ onto $f^k(\tau)$.
  \item Let $\sigma$ be an $n$-cell. Then $f^{-k}(\sigma)$ is equal to the union of all $(n+k)$-cells $\tau$ with $f^k(\tau)=\sigma$.
  \item The number of $n$-vertices is less than or equal to $m\deg(f)^n$, the number of $n$-edges is $m\deg(f)^n$, and the number of $n$-tiles is $2\deg (f)^n$.
  \item The $n$-edges are precisely the closures of the connected components of
       $f^{-n}(\mathcal C)\setminus f^{-n}(\post(f))$. The $n$-tiles are precisely the closures of the connected components of $\S^2\setminus f^{-n}(\mathcal C)$.
  \item Every $n$-tile is an $m$-gon, i.e., the number of $n$-edges and $n$-vertices contained in its boundary is equal to $m$.
\end{enumerate}
\end{lem}

\excise{
------------------------
Let $\sigma$ be an $n$-cell. Let $W^n(\sigma)$ be the union of the interiors of all $n$-cells intersecting with $\sigma$, and call $W^n(\sigma)$ the \emph{$n$-flower} of $\sigma$. In general, $W^n(\sigma)$ is not necessarily simply connected. The following lemma (from Lemma 7.2 in \cite{BMExpanding}) says that if $\sigma$ consists of a single $n$-vertex, then $W^n(\sigma)$ is simply connected.
\begin{lem} \label{flower}
Let $f\: \S^2\ra \S^2$ be a Thurston map, and let $\mathcal C$  be a Jordan curve containing $\post(f)$. If $\sigma$ is an $n$-vertex, then $W^n(\sigma)$ is simply connected. In addition, the closure of $W^n(\sigma)$ is the union of all $n$-tiles containing the vertex $\sigma$.
\end{lem}
One of the most important properties of $n$-flowers is that they build a connection between $n$-tiles of different Jordan curves due to the following lemma in \cite[Lemma 7.12]{BMExpanding}.
\begin{lem}\label{flowertile}
Let $\mathcal C$ and $\mathcal C'$ be Jordan curves in $\S^2$ both containing
$\post(f)$. Then there exists a number $M$ such that each $n$-tile for $(f, \mathcal C)$
is covered by $M$ $n$-flowers for $(f, \mathcal C')$.
\end{lem}
\begin{re}
The exact same proof for this lemma shows that for $n'\geq n$, there exists a number $M$ such that each $n'$-tile $(f, \mathcal C)$ is covered by $M$ $n$-flowers for $(f, \mathcal C')$.
\end{re}
-------------------------
}

We obtain a sequence of cell decompositions of $\S^2$ from a Thurston map and a Jordan curve on $\S^2$. It would be nice if the local degrees of the map $f$ at all the vertices were bounded, and this can be obtained by the assumption of no periodic critical points (see  \cite[Lemma 16.1]{BMExpanding}).
\begin{lem} \label{noperiodic}
Let $f : \S^2\ra \S^2$ be a branched covering map. Then f has no
periodic critical points if and only if there exists $N\in \N$ such that
\[\deg_{f^n}(p)\leq N,\]
for all $p\in \S^2$ and all $n\in \N$.
\end{lem}
Henceforth we assume that \emph{all Thurston maps have no periodic critical points}.

Let $f\:\S^2\ra \S^2$ be an expanding Thurston map and let $\mathcal{C}$ be a Jordan curve containing $\post(f)$.
\begin{de}\label{joinoppositesides}
A set $K\subseteq \S^2$ \emph{joins opposite sides} of $\mathcal C$ if $\#$$\post(f)\geq 4$
and $K$ meets two disjoint $0$-edges, or if $\#$$\post(f) = 3$ and $K$ meets all
three $0$-edges.
\end{de}

Let $D_n=D_n(f,\mathcal C)$ be the minimum number of $n$-tiles needed to join opposite sides of a Jordan curve $\mathcal{C}$. More precisely,
\begin{eqnarray}\label{defdn}
D_n =\min\{N\in \N: \mbox{ there exist $n$-tiles } X_1, . . . ,X_N \mbox{ such that }  \\
\bigcup_{j=1}^N X_j\mbox{ is connected and joins opposite sides of }\mathcal C\}. \nonumber
\end{eqnarray}
Of course, $D_n$ depends on $f$ and $\mathcal C$.


Let $f$ be an expanding Thurston map. For any two Jordan curves $\mathcal{C}$ and $\mathcal{C'}$ with $\post(f)\subset \mathcal C,\mathcal C'$, inequality (17.1) in \cite{BMExpanding} states that there exists a constant $c>0$ such that for all $n>0$,
\[ \frac1{c}D_n(f,\mathcal C)\leq D_n(f,\mathcal C')\leq c D_n(f,\mathcal C).\]
Proposition 17.1 in \cite{BMExpanding} says that:
\begin{pro} \label{expansionfactor}
For an expanding Thurston map $f\: \S^2\ra \S^2$, and a Jordan curve $\mathcal C$ containing $\post(f)$,
the limit
\[\Lambda_0=\Lambda_0(f):=\lim_{n\ra\infty}D_n(f,\mathcal{C})^{1/n}\] exists and is independent of $\mathcal C$.
\end{pro}

We call $\Lambda_0(f)$ the \emph{combinatorial expansion factor} of $f$.

Proposition 17.2 in \cite{BMExpanding} states that:
\begin{pro}
If $f\:\S^2\ra \S^2$ and $g\:\S^2_1\ra \S^2_1$ are expanding Thurston maps that are topologically conjugate, then $\Lambda_0(f)=\Lambda_0(g)$.
\end{pro}

\begin{de}
Let $f \: \S^2 \ra \S^2$ be an expanding Thurston map, and let
${\mathcal C}\subseteq \S^2$ be a Jordan curve containing $\post(f) $. Let $x, y \in \S^2$.
For $x \not= y$ we define
\begin{eqnarray*}
m_{f,\mathcal C}(x, y) = \min\{n\in \N_0 :\mbox{ there exist disjoint $n$-tiles }X \mbox{ and } Y \\
\mbox{ for }(f, \mathcal C)  \mbox{ with } x\in X \mbox{ and } y\in Y \}.
\end{eqnarray*}
If $x = y$, we define $m_{f,\mathcal C}(x, x)= \infty$.
\end{de}

The minimum in the definition above is always obtained since the
diameters of $n$-tiles go to $0$ as $n\ra \infty$. We usually drop one or both
subscripts in $m_{f,\mathcal C}(x, y)$ if $f$ or $\mathcal C$ is clear from the context. If
we define for $x,y\in \S^2$ and $x \not= y$,
\begin{eqnarray*}
m'_{f,\mathcal C}(x, y) = \max\{n\in \N_0 : \mbox{ there exist nondisjoint $n$-tiles }X \mbox{ and } Y  \\
\mbox{ for } (f, \mathcal C) \mbox{ with } x\in X \mbox{ and } y\in Y \},
\end{eqnarray*}
then $m_{f,\mathcal C}$ and $m'_{f,\mathcal C}$ are essentially the same up to a constant (see Lemma 8.6 (v) in \cite{BMExpanding}).
\begin{lem} \label{twom}
Let $m_{f,\mathcal C}$ and $m'_{f,\mathcal C}$ as defined above. There exists a constant $k>0$, such that for any $x,y\in \S^2$ and $x\not=y$,
\[ m_{f,\mathcal C}(x,y)-k\leq m'_{f,\mathcal C}(x,y)\leq m_{f,\mathcal C}(x,y)+1.\]
\end{lem}

\begin{de}\label{visual}
Let $f \: \S^2\ra \S^2$ be an expanding Thurston map and
$d$ be a metric on $\S^2$. The metric $d$ is called a \emph{visual metric} for $f$ if there
exists a Jordan curve $\mathcal C\subseteq \S^2$ containing $\post(f) $, constants $\Lambda > 1$ and  $C \geq 1$ such that
\[\frac1{C}\Lambda^{-m_{f,\mathcal C}(x, y)} \leq d(x, y) \leq C\Lambda^{-m_{f,\mathcal C}(x, y)}\]
for all $x, y \in \S^2$.
\end{de}

Proposition 8.9 in \cite{BMExpanding} states that for any expanding Thurston map $f\:\S^2\ra \S^2$, there exists a visual metric for $f$, which induces the standard topology on $\S^2$. Lemma 8.10 in the same paper gives the following characterization of visual metrics.
\begin{lem} \label{charvisual}
Let $f \: \S^2 \ra \S^2$ be an expanding Thurston map. Let $\mathcal C\subseteq \S^2$
be a Jordan curve containing $\post(f)$, and $d$ be a visual metric for $f$ with
expansion factor $\Lambda > 1$. Then there exists a constant $C > 1$ such that
\begin{enumerate}
  \item $d(\sigma,\tau)\geq (1/C)\Lambda^{-n}$ whenever $\sigma$ and $\tau$ are disjoint $n$-cells,
  \item $(1/C)\Lambda^{-n}\leq \diam (\tau)\leq C\Lambda^{-n}$ for $\tau$ as any $n$-edge or $n$-tile.
\end{enumerate}
Conversely, if $d$ is a metric on $\S^2$ satisfying conditions $(1)$ and $(2)$
for some constant $C>1$, then $d$ is a visual metric with expansion
factor $\Lambda > 1$.
\end{lem}
The combinatorial expansion factor $\Lambda_0(f)$ (defined after Proposition \ref{expansionfactor}) is the supremum of the expansion factors for all the visual metric for $f$ (see Theorem 1.7 in \cite{BMExpanding}).
\begin{thm} \label{expansionvisualsup}
Let $f \: \S^2 \ra \S^2$ be an expanding Thurston map with combinatorial expansion factor $\Lambda_0(f)$.  Then
\begin{eqnarray*}
\Lambda_0:=\sup\{\Lambda>1\:& \mbox{there exists a visual metric for } f \\
 & \mbox{ with expansion factor }\Lambda \}.
\end{eqnarray*}
\end{thm}

\section{Gromov Hyperbolic Spaces}
\noindent
In this section, we review the definitions of Gromov hyperbolic spaces and the asymptotic upper curvature for Gromov hyperbolic spaces.

Let us first review some basic facts about Gromov hyperbolic spaces. We refer the reader to \cite{BSElements} as a general source on Gromov hyperbolic spaces. Let $(X,d)$ be a geodesic metric space. For any points $x,y,p\in X$, the \emph{Gromov product} $(x,y)_p$ of $x$ and $y$ with respect to base point $p$ is defined as
\begin{eqnarray} \label{gproduct}
  (x,y)_p := \frac12 \left[d(x,p)+d(y,p)-d(x,y) \right].
\end{eqnarray}
The space $X$ is called \emph{$\delta$-hyperbolic} (or Gromov hyperbolic) for some $\delta\geq 0$ if there exists a base point $p\in X$, such that for all $x,y,z\in X$ we have
\begin{eqnarray} \label{trianglein}
  (x,y)_p\geq \min\{(x,z)_p,(z,y)_p\}-\delta.
\end{eqnarray}
If this inequality holds for some base point $p\in X$, then it also holds for any other $p'\in X$ with $\delta$ being replaced by $2\delta$.

Let $(X,d)$ be a Gromov hyperbolic metric space with a fixed base point $p\in X$. A sequence of points $\{x_i\}\subseteq X$ \emph{converges to infinity} if
\begin{eqnarray*}
\lim_{i,j\ra\infty} (x_i,x_j)_p=\infty.
\end{eqnarray*}
This property of a sequence $\{x_i\}$ does not depend on the base point $p\in X$. We say two sequences converging to infinity $\{x_i\}$ and $\{x_i'\}$ are \emph{equivalent} if
\begin{eqnarray*}
\lim_{i\ra\infty} (x_i,x_i')_p=\infty.
\end{eqnarray*}
The \emph{boundary at infinity} $\partial_{\infty} X$ of $X$ is defined to be the set of equivalence classes of sequences of points converging to infinity. One can also define the \emph{Gromov product} for points $\xi,\xi'\in \partial_{\infty}X$ and $p\in X$ as
\begin{eqnarray*}
(\xi,\xi')_p:=\inf \liminf_{i\ra\infty}(x_i,x_i')_p
\end{eqnarray*}
where the infimum is taken over all sequences $\{x_i\}\in \xi$ and $\{x_i'\}\in \xi'$. Here $(\xi,\xi')_p=\infty$ if and only if $\xi=\xi'$.

A metric $\rho$ on the boundary at infinity  $\partial_{\infty} X$ of a Gromov hyperbolic space $X$ is called \emph{visual} if there exist $p\in X$, $\Lambda>1$ and $k\geq 1$ such that for all $\xi,\xi'\in \partial_{\infty}X$, we have that
\begin{eqnarray} \label{visualg}
\frac1{k}\Lambda^{-(\xi,\xi')_p}\leq \rho(\xi,\xi') \leq k \Lambda^{-(\xi,\xi')_p}.
\end{eqnarray}
We call the constant $\Lambda$ in this inequality the \emph{expansion factor} of the visual metric $\rho$.
Recall that we also defined a visual metric for an expanding Thurston map (see Definition 2.11).
When it is not clear from context, we will refer to the visual metric defined in \eqref{visualg} as a `visual metric in the Gromov hyperbolic sense'.

Given two metric spaces $(X,d_X)$ and $(Y,d_Y)$,  a map $f\:X\ra Y $ is called a \emph{quasi-isometry} if there are constants $\lambda\geq 1$ and $k\geq 0$ such that for all $x,x'\in X$
\[\frac1{\lambda}d_X(x,x')-k\leq d_Y(f(x),f(x'))\leq \lambda d_X(x,x')+k \] and for all $y\in Y$,
\[\inf_{x\in X} d_Y(f(x),y)\leq k. \]
If $\lambda=1$, we call the map $f$ a \emph{rough-isometry}. We say that the spaces $X$ and $Y$ are \emph{quasi-isometric (rough-isometric)} if there is a quasi-isometry (rough-isometry) between them.

In \cite{BFAsymptotic}, Bonk and Foertsch introduced the notion of upper curvature bounds for Gromov hyperbolic spaces up to rough-isometry (see \cite[Definition 1.1 and 1.2]{BFAsymptotic}).
\begin{de}
Let $\kappa\in [-\infty,0)$. We call a metric space $X$ an \emph{AC$_u(\kappa)$-space} if there exists $p\in X$ and a constant $c\geq 0$ such that for all $x,x'\in X$ and all finite sequences $x_0=x,x_1,\ldots,x_n=x'$ in $X$ with $n>0$,
\begin{eqnarray}\label{acspaceeq}
(x,x')_p\geq \min_{i=1,\ldots,n}(x_{i-1},x_i)_p-\frac1{\sqrt{-\kappa}}\log n-c.
\end{eqnarray}
Here we use the convention $1/\sqrt{\infty}=0$.
We call
\begin{eqnarray*}
K_u(X):=\inf\{\kappa\: X \mbox{ is an AC$_u(\kappa)$-space}\in [-\infty,0) \}
\end{eqnarray*}
the \emph{asymptotic upper curvature} of $X$.
\end{de}
Rough-isometric Gromov hyperbolic spaces have the same asymptotic upper curvature since under rough-isometries, Gromov products only change by a fixed additive amount, which can be absorbed in the constant $c$ in \eqref{acspaceeq}.

The asymptotic upper curvature is related to the expansion factors of visual metrics in Gromov hyperbolic spaces, due to the following theorem \cite[Theorem 1.5]{BFAsymptotic}.
\begin{thm} \label{acku}
Let $X$ be a Gromov hyperbolic metric space. If there exists a visual metric on $\partial_{\infty}X$ with expansion factor $\Lambda>1$, then $X$ is an AC$_u(\kappa)$-space with $\kappa=-\log^2\Lambda$. Conversely, if $X$ is an AC$_u(\kappa)$-space, then for every $1<\Lambda<e^{\sqrt{-\kappa}}$, there exists a visual metric on $\partial_{\infty}X$ with expansion $\Lambda$.
In particular,
\[K_u(X)=-\log^2\Lambda_0,\]where
\begin{eqnarray*}
\Lambda_0:=\sup\{\Lambda>1\:& \mbox{there exists a visual metric on } \partial_{\infty}X \\
 & \mbox{ with expansion factor }\Lambda \}.
\end{eqnarray*}
\end{thm}

\section{Tile Graphs}
\noindent
In this section, we construct graphs for expanding Thurston maps. We prove that these graphs are Gromov hyperbolic and their boundary at infinity can be identified with $\S^2$. This construction should be compared to the construction of graphs associated to finite branched coverings in Section 3.2 and 3.3 in \cite{HPCoarse}.

Let $f \: \S^2\ra \S^2$ be an expanding Thurston map,
and $\mathcal C\subset \S^2$ be a Jordan curve such that
post$(f)\subset \mathcal C$.
Recall that there is a natural sequence of cell decompositions ${\mathcal D}^n(f,\mathcal C)$
on $\S^2$ whose $1$-skeletons are the pull-backs of the Jordan curve $\mathcal C$ under $f^n$ (see Section 2).
Proposition 8.9 in \cite{BMExpanding} states that there exists a visual metric $d$ for $f$ with expansion factor $\Lambda$ for some $\Lambda>1$.

We define a graph by the cell decompositions of $(f, \mathcal C)$ as follows.
Let
\[V=V(f,\mathcal C)\]
be the set of all tiles in the cell decompositions $\D^n(f,\mathcal C)$ of $(f, \mathcal C)$ for $n\geq -1$, where  $\D^{-1}(f,\mathcal C)$ contains a single $(-1)$-tile $\S^2$.
Let $V$ be the set of vertices of the graph. Define the edge set $E$ as follows:
there is an edge between the two vertices $X^n,Y^m\in V$, which we indicate by the notation
$X^n\sim Y^m$ if for the underlying tiles we have
\[|n-m|\leq 1 \mbox{ and } X^n\cap Y^m\not= \emptyset.\]
We call the graph
\[\G(f,\mathcal C):=G(V,E)\]
the \emph{tile graph} of $(f, \mathcal C)$. We usually drop one or both
parameters in $\G(f,\mathcal C)$ if $f$ or $\mathcal C$ are clear from the context.
We call \[\ell\: V \ra \Z\]
the \emph{level function}, where
for an $n$-tile $X^n$, we have $\ell(X^n)=n$.

If $X\cap Y=\emptyset$, let
\begin{align*}
\bar m_{f,\mathcal C}(X,Y) :=\max\{m\in \N_{-1}\:&\mbox{
there exist non-disjoint $m$-tiles $X^m$ and} \\
&\mbox{$Y^m$, such that } X\cap X^m\not= \emptyset,
Y\cap Y^m\not= \emptyset\};
\end{align*}
if $X\cap Y\not=\emptyset$, let
\begin{align*}
\bar m_{f,\mathcal C}(X,Y) :=\infty. 
\end{align*}
Here we assume that the $\infty$-tile is the empty set.
For $X,Y\in \G$, define
\begin{eqnarray} \label{mG}
m(X,Y)=m_{f,\mathcal C}(X,Y)=\min\{\ell(X),\ell(Y), \bar m_{f,\mathcal C}(X,Y)\}.
\end{eqnarray}
The tile graph $\G$ is path connected since any tile can be connected to the $(-1)$-tile $\S^2$.
We give $\G$ the path metric $\eta$. Notice that $\G$ is a
geodesic space under this metric. The distance of $X\in V$ to the base point $\S^2$ is
\[\eta(X,\S^2)= \ell(X)+1.\]
For $X,Y\in \G$, we let
\begin{eqnarray} \label{gp}
(X,Y)&:=& (X,Y)_{X^{-1}} = (X,Y)_{\S^2} \nonumber \\
&=& 1/2[\eta(X,\S^2)+\eta(Y,\S^2)-\eta(X,Y)]\\
&=&1/2[\ell(X)+\ell(Y)-\eta(X,Y)]+1, \nonumber
\end{eqnarray} be the
\emph{Gromov product} of $X$ and $Y$ with respect to $X^{-1}=\S^2$.

In the following, we are going to prove that the tile graph $\G$ equipped with the path metric $\eta$ is a Gromov
hyperbolic space.

\excise{
\begin{eg}
We define a map $f$ as follows (see the picture below): we glue along the boundary of two unit squares $[0,1]^2$, and get a pillow-like space which is homeomorphic to $\widehat{\C}$; we color one of the squares black and the other white; we divide each of the squares into 4 smaller squares of half the side length, and color them with black and white in checkerboard fashion; we map one of the small black pillows to the bigger black pillows by Euclidean similarity, and extend the map to the whole pillow-like space by reflection. In fact, the map $f$ is an expanding Thurston map (see Example 4.13 in \cite{YinLattes}), and the postcritical set $\post (f)$ consists of the four common corner points of the two big squares.
\begin{center}
\mbox{ \scalebox{0.7}{\includegraphics{2by2.eps}}}
\end{center}
Let $\mathcal C$ be the common boundary of the two big squares, then $\mathcal C$ contains $\post (f)$.
On the $n$-th level, the set of $n$-tiles corresponding to
\end{eg}
}

\begin{lem} \label{mdiam}
There exists a constant $C>1$ such that for any tiles $X,Y\in \G$,
\[\frac1{C}\Lambda^{-m(X,Y)}\leq \diam (X\cup Y) \leq  C \Lambda^{-m(X,Y)}.\]
\end{lem}
Here and in the following, the diameter function $\diam(\cdot)$ is with respect to the visual metric $d$ on $\S^2$.

\begin{proof}
Let $m=m(X,Y)$, and let $X^m, Y^m$ be non-disjoint $m$-tiles such that
\[X\cap X^m\not=\emptyset \mbox{ and } Y\cap Y^m\not=\emptyset.\]
We have that
\begin{eqnarray*}
\diam (X\cup Y) &\leq &\diam (X)+\diam (X^m)+\diam (Y^m)+\diam (Y)\\
&\leq& 4 C' \Lambda^{-m},
\end{eqnarray*}
where $C'>1$ is the same as the constant in Lemma \ref{charvisual}, which only depends on $f$ and $\mathcal C$.
Let $\bar m= \bar m_{f,\mathcal C}(X,Y)$, and let $X^{\bar m+1}, Y^{\bar m+1}$ be disjoint $(\bar m+1)$-tiles such that
\[X\cap X^{\bar m+1}\not=\emptyset,\quad Y\cap Y^{\bar m+1}\not=\emptyset.\] Then
\begin{eqnarray*}
\diam (X\cup Y) &\geq& \max\{ \diam (X), \diam (Y), d(X^{\bar m+1}, Y^{\bar m+1}) \} \\
&\geq & \frac1{C'}\max\{\Lambda^{-\ell(X)}, \Lambda^{-\ell(Y)}, \Lambda^{-\bar m}\} \\
&\geq &\frac1{C'}\Lambda^{-\min\{\ell(X),\ell(Y),\bar m\}}\\
&= &\frac1{C'}\Lambda^{-m},
\end{eqnarray*}
where $C'>1$ is the same $C$ as in Lemma \ref{charvisual}, which only depends on $f$ and $\mathcal C$. Let $C=4C'$, and the lemma follows.
\end{proof}

\begin{lem} \label{pdiam}
There exists a constant $k\geq 1$ such that for any tiles $X,Y\in \G$,
\[\diam (X\cup Y) \leq k\Lambda^{-(X,Y)}.\]
\end{lem}

\begin{proof}
Let $\eta=\eta(X,Y)$. Pick any path $X_0=X, X_1,\ldots, X_{\eta}=Y$. Then
\begin{eqnarray*}
\diam (X\cup Y)& \leq & \sum_{i=0}^{\eta}\diam(X_i)\\
&\leq & C \sum_{i=0}^{\eta}\Lambda^{-\ell(X_i)} \\
&\leq & C \min_{0\leq l\leq \eta}\left\{\sum_{i=0}^{l}\Lambda^{-\ell(X)+i}+ \sum_{i=l+1}^{\eta}\Lambda^{-\ell(Y)+(\eta-i)}\right\}\\
&\leq & \frac{C\Lambda}{\Lambda-1} \min_{0\leq l\leq \eta}\left\{\Lambda^{-\ell(X)+l}+ \Lambda^{-\ell(Y)+(\eta-l)}\right\}.
\end{eqnarray*}
Notice that on the right hand-side the minimum is obtained when the two exponents  of $\Lambda$ are the same:
\[ -\ell(X)+l= -\ell(Y)+(\eta-l),\]
so we let
\[l=\left[\frac12(\ell(X)-\ell(Y)+\eta)\right]\]
be the integer part of $\frac12(\ell(X)-\ell(Y)+\eta)$. Hence, we have
\begin{eqnarray*}
\diam (X\cup Y) &\leq & \frac{2C\Lambda}{\Lambda-1} \Lambda^{-[1/2(\ell(X)+\ell(Y)-\eta)]} \\
&\leq &k\Lambda^{-(X,Y)},
\end{eqnarray*}
where $C>1$ is the same $C$ as in Lemma  \ref{charvisual}, and
\[k=\frac{2C\Lambda^3}{\Lambda-1}\] also only depends on $f$.
\end{proof}

\begin{pro} \label{mp}
There exists a constant $C'>0$, such that for any tiles $X,Y\in \G$,
\begin{eqnarray*}
m(X,Y)-1\leq (X,Y)\leq m(X,Y)+C'.
\end{eqnarray*}
\end{pro}

\begin{proof}
By Lemma \ref{mdiam} and Lemma \ref{pdiam}, we have that
\[ \frac1{C}\Lambda^{-m(X,Y)}\leq \diam (X\cup Y) \leq k\Lambda^{-(X,Y)}\] for some constants $C,k>1$ which only depend on $f$. Hence, there exists a constant $C'>0$, such that for any tiles $X,Y\in \G$,
\[(X,Y)\leq m(X,Y)+C'.\]

For the other inequality, let $m=m(X,Y)$, and let $X^m, Y^m$ be non-disjoint $m$-tiles such that
\[X\cap X^m\not=\emptyset,\quad Y\cap Y^m\not=\emptyset.\]
So \[\eta(X,X^m)\leq \ell(X)-m+1 ,\]and
\[\eta(Y,Y^m)\leq \ell(Y)-m+1 .\]
By the triangle inequality, we have that
\begin{eqnarray*}
\eta(X,Y)&\leq& \eta(X,X^m)+\eta(X^m,Y)\\
&\leq & \eta(X,X^m)+\eta(Y^m,Y)+1\\
&\leq & (\ell(X)-m)+(\ell(Y)-m)+3.
\end{eqnarray*}
Hence, we obtain that
\begin{eqnarray*}
(X,Y)&= &(X,Y)_{X^{-1}} =1/2[\ell(X)+\ell(Y)-\eta(X,Y)]+1\\
&\geq & 1/2[\ell(X)+\ell(Y)-(\ell(X)-m)-(\ell(Y)-m)-3]+1\\
&\geq &m(X,Y)-1.
\end{eqnarray*}
\end{proof}

\excise{
----------------------------------
\begin{cor}
There exists a constant $C'>0$, such that for any tiles $X,Y\in \G$, we have
\[  (X,Y)\leq m(X,Y)+C'.\]
\end{cor}
\begin{proof}
By the Lemma \ref{mdiam} and Lemma \ref{pdiam}, we have
\[ \frac1{k}\Lambda^{-m(X,Y)}\leq \diam (X\cup Y) \leq k'\Lambda^{-(X,Y)},\] for some constant $k,k'>1$ which only depends on $f$. The corollary follows easily.
\end{proof}

\begin{lem}
\[ m(X,Y)\leq (X,Y).\]
\end{lem}

\begin{proof}
Let $m=m(X,Y)$, and let $X^m, Y^m$ be non-disjoint $m$-tiles such that
\[X\cap X^m\not=\emptyset,\quad Y\cap Y^m\not=\emptyset.\]
By triangle inequality, we have
\begin{eqnarray*}
\eta(X,Y)&\leq& \eta(X,X^m)+\eta(X^m,Y)\\
&\leq & \eta(X,X^m)+\eta(Y^m,Y)+1\\
&\leq & (\ell(X)-m)+(\ell(Y)-m)+1.
\end{eqnarray*}
Hence,
\begin{eqnarray*}
(X,Y)&= &(X,Y)_{X^{-1}} =1/2[\ell(X)+\ell(Y)-\eta(X,Y)]+1\\
&\geq & 1/2[\ell(X)+\ell(Y)-(\ell(X)-m)-(\ell(Y)-m)-1]+1\\
&\geq &m(X,Y).
\end{eqnarray*}
\end{proof}
-------------------------------}

\begin{lem} \label{triangleineq}
There exists a number $c\geq 0$ such that for any tiles $X,Y,Z\in \G$,
\[m(X,Y)\geq \min\{m(X,Z),m(Y,Z)\} -c .\]
\end{lem}

\begin{proof}
For any $X,Y,Z\in \G$,
\begin{eqnarray*}
\diam(X\cup Y)&= &\max\{d(x,y),d(x,x'),d(y,y')\: x,x'\in X, y,y'\in Y\}\\
&\leq & \max\{d(x,z)+d(z,y),d(x,x'),d(y,y')\: \\
&& \hspace{1.5cm} x,x'\in X, y,y'\in Y, z\in Z\}\\
&\leq & \max\{d(x,z),d(x,x')\: x,x'\in X,z\in Z\}\\
&& +\max\{d(z,y),d(y,y')\: y,y'\in Y,z\in Z\}\\
&\leq & \diam(X\cup Z)+\diam(Z \cup Y),
\end{eqnarray*}
and so
\begin{eqnarray} \label{diamtri}
\diam(X\cup Y)&\leq & 2\max\{ \diam(X\cup Z),\diam(Z \cup Y)\}.
\end{eqnarray}

By Lemma \ref{mdiam}, there exists a constant $k>1$, such that for any $X,Y\in \G$,
\begin{eqnarray*}
\frac1{k}\Lambda^{-m(X,Y)}\leq \diam (X\cup Y) \leq k \Lambda^{-m(X,Y)}.
\end{eqnarray*}
Hence, by the inequalities above and inequality \eqref{diamtri}, we have that
\begin{eqnarray*}
m(X,Y)&\geq & -\log_{\Lambda}\big(k\diam(X\cup Y)\big) \\
&\geq & -\log_{\Lambda}\Big(2k\max\big\{ \diam(X\cup Z),\diam(Z \cup Y) \big\}\Big)\\
&\geq & \min\Big\{ -\log_{\Lambda}\big(2k\diam(X\cup Z)\big),-\log_{\Lambda}\big(2k\diam(Z \cup Y)\big) \Big\}\\
&\geq & \min\{m(X,Z),m(Y,Z)\} -c
\end{eqnarray*}
for some $c\geq 0$ that only depends on $f$.
\end{proof}

\begin{thm} \label{gh}
Let $f\: \S^2 \ra S^2$ be an expanding Thurston map 
and let $\mathcal C \subset \S^2$ be a Jordan curve containing $\post(f)$. Then the tile graph $\G(f,\mathcal C)$ equipped with the path metric $\eta$ is a Gromov hyperbolic space.
\end{thm}

\begin{proof}
For any tiles $X,Y\in \G$, by Proposition \ref{mp}, the Gromov product $(X,Y)$ defined in equation \eqref{gp} is equal to $m(X,Y)$ up to a constant which only depends on $f$. So by Lemma \ref{triangleineq}, there exists a constant $c'>0$, such that for any tiles $X,Y,Z\in \G$,
\[(X,Y)\geq \min\{(X,Z), (Y,Z)\} -c' .\] Therefore, the graph $G(f,\mathcal C)$ equipped with the path metric $\eta$ is a Gromov hyperbolic space.
\end{proof}

\begin{re}
In the proofs of Proposition \ref{mp} and Lemma \ref{triangleineq}, we used visual metrics as a bridge to connect $m(\cdot,\cdot)$ and the Gromov product $(\cdot,\cdot)$. This idea is contained in \cite{BPCohomologie}. Theorem \ref{gh}, Proposition \ref{mp} and Lemma \ref{triangleineq} can also be proved combinatorially without using visual metrics.
\end{re}

\begin{pro} \label{roughgg}
For any Jordan curves $\mathcal C$ and $\mathcal C'$ containing $\post (f)$, the tile graphs $\G=\G(f,\mathcal C)$ and $\G'=\G(f,\mathcal C')$ equipped with path metric respectively are rough-isometric.
\end{pro}

\begin{proof}
By equation \eqref{gp}, for any $X,Y\in \G(f,\mathcal C)$, we have
\begin{eqnarray} \label{dp}
  \eta(X,Y) &=& \ell(X)+\ell(Y)+2-2(X,Y).
\end{eqnarray}
We have similar relations for the path metric $\eta'$ of $\G'$. Let $m=m_{f,\mathcal C}$ and $m'=m_{f,\mathcal C'}$ as defined in equation \eqref{mG}.
We know that $m(X,Y)$ and $(X,Y)$ are equal up to a constant that only depends on $f$ by Proposition \ref{mp}. So
if we can show that there exists a level-preserving bijection $g\:\G\ra \G'$ and a constant $\lambda\geq 0$, such that for any $X,Y\in \G$,
\begin{eqnarray*}
  m(X,Y)-\lambda\leq m'(g(X),g(Y))\leq m(X,Y)+\lambda,
\end{eqnarray*}
then by equation \eqref{dp}, the map $g$ will be a rough isometry between the path metrics of $\G$ and $\G'$.

Fix $p\in {\rm post}(f)$. We will define
\[g\: \G(f,\mathcal C)\ra \G(f,\mathcal C')\]
by specifying a bijection between $n$-tiles of $(f,\mathcal C)$ and $(f,\mathcal C')$ for all $n\geq -1$. For $n=-1$, let
\[g(\S^2)=\S^2.\]
For $n\geq 0$, and for any $q\in f^{-n}(p)$, we claim that there exists a bijection $g_{n,q}$ between $n$-tiles of $(f,\mathcal C)$ containing $q$ and $n$-tiles of $(f,\mathcal C')$ containing $q$,
\begin{equation*}
   g_{n,q}\: \Big\{n\mbox{-tile } X\in \D^n(f,\mathcal C): q\in X \Big\}\ra \Big\{n\mbox{-tile } X'\in \D^n(f,\mathcal C'): q\in X' \Big\}.
\end{equation*}
Indeed, the number of tiles containing $q$ is equal to twice the degree of $f^n$ at $q$, and this justifies the existence of the bijection $g_{n,q}$. Since every $n$-tile contains exactly one point in $f^{-n}(p)$, we get a bijection of all $n$-tiles by $g_{n,q}$ for $q\in f^{-n}(p)$.

For any $X, Y\in \G$, let $X',Y' \in \G'$ be their images under $g$. It follows from the definition of $g$ that
\[X\cap X'\not=\emptyset \mbox{ and } Y\cap Y'\not=\emptyset.\]
Now we are going to show that there exists $k\geq 1$, such that for any $X,Y\in \G$,
\begin{eqnarray*}
\frac1{k}\diam (X'\cup Y') \leq \diam (X\cup Y)\leq k\diam (X'\cup Y').
\end{eqnarray*}
Let $m=m(X,Y)$. We have that
\begin{eqnarray*}
\diam (X\cup Y)&\leq &\diam(X) + \diam (X'\cup Y') + \diam (Y)\\
&\leq & C^2 \diam (X')+C^2\diam(Y') + \diam(X'\cup Y')\\
&\leq & (C^2+1) \diam (X'\cup Y'),
\end{eqnarray*}
where $C>1$ is the same $C$ as in Lemma \ref{charvisual}, which only depends on $f$, $\mathcal C$ and $\mathcal C'$.
This implies that
\[\diam (X\cup Y)\leq k\diam (X'\cup Y'), \]for some $k>1$ only depending on $f$. Similarly, we get that
\[\diam (X'\cup Y')\leq k\diam (X\cup Y). \]

Since $\diam (X\cup Y)$ and $\Lambda^{-m(X,Y)}$ are the same up to a scaling by Lemma \ref{mdiam}, there exists a constant $\lambda>0$, such that
\[m(X,Y)-\lambda\leq m'(g(X),g(Y))\leq m(X,Y)+\lambda \] for all $X,Y\in \G(f,\mathcal C)$.
\end{proof}

\begin{re}
In the proof of Proposition \ref{roughgg}, the bijective rough-isometry $g$ between tile graphs of two different Jordan curves induces a bijection $g_{\infty}$ on the boundary at infinity of these two tile graphs.
\end{re}

\begin{re}
Theorem \ref{gh} and Proposition \ref{roughgg} should be compared to Theorem 3.3.1 in \cite{HPCoarse}. More specifically, Ha{\"{\i}}ssinsky and Pilgrim introduce a collection of graphs for a more general notion of an expanding Thurston map and prove that they are Gromov hyperbolic, and quasi-isometric to each other.
One should be able to prove that the tile graph $\G(f, {\mathcal C})$ is quasi-isometric to a graph from Theorem 3.3.1 in \cite{HPCoarse}, and hence deduce an alternative proof of Theorem \ref{gh}.
\end{re}

\begin{pro} \label{samevisual}
The boundary at infinity $\partial_{\infty}\G$ of a tile graph $\G(f,\mathcal C)$ can be identified with $\S^2$.
Under this identification, a metric $d$ is a visual metric on $\S^2$ with respect
to the expanding Thurston map $f$ if and only if $d$ is a visual metric on $\partial_{\infty}\G$ (in the sense of Gromov hyperbolic spaces).
\end{pro}

Here the metric $d$ on $\partial_{\infty}\G$ means the pull-pack metric of $d$ under the identification.

\begin{proof}
Let $d$ be a visual metric with expansion factor $\Lambda$ of $\S^2$ with respect to $f$.

For any  sequence $\{X_n\}$ converging to $\infty$
\[\lim_{i,j\ra \infty} (X_i,X_j)=\infty,\] we have a filtration
\begin{eqnarray*}
\bigcup_{i={1}}^{\infty} X_i\supset \bigcup_{i={2}}^{\infty} X_i\supset\ldots \supset\bigcup_{i={n}}^{\infty} X_i \supset \bigcup_{i=n+1}^{\infty} X_i\supset\ldots
\end{eqnarray*}
with
\[ \diam \left(\bigcup_{i={n}}^{\infty} X_i\right)\ra 0\mbox{ as } n\ra \infty.\]
Hence, there exists a limit point $x\in \S^2$ such that for any $\epsilon>0$, there exists $N>0$ such that for all $n>N$,
\begin{eqnarray} \label{inepsilonn}
\bigcup_{i={n}}^{\infty} X_i\subset N_{\epsilon}(x),
\end{eqnarray}
where $N_{\epsilon}(x)$ is an $\epsilon$-neighborhood of $x$ in $\S^2$, i.e.,
\[  X_n\subset N_{\epsilon}(x),\] or
\begin{eqnarray*}
d(x,X_n)< \epsilon.
\end{eqnarray*}
We claim that the limit point is unique. Indeed, if there exists $y\in \S^2$ also satisfying \eqref{inepsilonn}, then
\[d(x,y)\leq d(x,\diam(X_n))+d(y,\diam(X_n))\ra 0\mbox{ as } n\ra\infty.\]
Hence, $x=y$. 
Let $\{Y_n\}$ be an sequence  converging to infinity equivalent to $\{X_i\}$, i.e.,
\[\lim_{i\ra \infty} (X_i,Y_i)=\infty.\]
We claim that the limit point of $\{Y_n\}$ is $x$. Indeed, by Lemma \ref{pdiam}, we have
\[d(x,Y_n)\leq d(x,X_n)+d(Y_n,X_n)\leq d(x,X_n)+k{\Lambda}^{-(Y_n,X_n)}\ra 0 \]
as $n$ goes to infinity since $(Y_n,X_n)\ra \infty$. Hence, any two equivalent sequences converging to infinity have the same limit point, and we can assign a limit point to an equivalence class of sequences converging to infinity.

We define
\[h\: \partial_{\infty}\G\ra \S^2\]
by mapping any equivalence class of sequences  converging to infinity to its limit point. For any $x\in \S^2$, there exists $X_i$ with $\ell(X_i)=i$ containing $x$, for any $i\geq -1$. Then by Lemma \ref{pdiam}, we have that
\begin{eqnarray*}
(X_i,X_j)&\geq& -\log_{\Lambda}\diam(X_i\cup X_j)+\log k\\
&\geq&-\log_{\Lambda} \big(\min\{ \diam(X_i),\diam(X_j)\} \big) +\log k \ra \infty
\end{eqnarray*}
as $i,j\ra\infty$, where $k\geq 1$ is a constant as in Lemma \ref{pdiam}.
So $\{X_i\}$ is a converging sequence with limit point $x$. Hence, the map $h$ is surjective.
In order to prove the injectivity, for any two 
sequences converging to infinity $\{X_i\}$ and $\{Y_i\}$, we let $x$ and $y$ be their limit points respectively. If $x=y$, then
\[\diam(X_n\cup Y_n) \ra 0 \mbox{ as } n\ra \infty.\]
So by Lemma \ref{mdiam} and \ref{mp},
\[(X_n,Y_n) \geq m(X_n,Y_n)-1 \geq -\log_{\Lambda}\big(C\diam(X_n\cup Y_n)\big) -1\ra \infty \]
as $n$ goes to infinity, which implies that $\{X_i\}$ and $\{Y_i\}$ are equivalent.
Hence, 
$h$ is injective.

We only need to show that that there exists a constant $C>0$ such that for any $\xi,\xi'\in \partial_{\infty}\G$, $x=h(\xi)$ and $y=h(\xi')$,
\[\frac1{C} \Lambda^{-(\xi,\xi')}\leq d(x,y)\leq C\Lambda^{-(\xi,\xi')}. \]
Pick any $\{X_n\}\in \xi$ and $\{Y_n\}\in \xi'$. By Lemma \ref{mdiam}
\begin{eqnarray*}
\frac1{C}\Lambda^{-(X_n,Y_n)}\leq \diam(X_n\cup Y_n) \leq C \Lambda^{-(X_n,Y_n)}.
\end{eqnarray*}
Taking the limit superior, we get
\begin{eqnarray*}
\frac1{C}\limsup_{n\ra\infty} \Lambda^{-(X_n,Y_n)}\leq \lim_{n\ra\infty}\diam(X_n\cup Y_n)
\leq C\limsup_{n\ra\infty}  \Lambda^{-(X_n,Y_n)}.
\end{eqnarray*}
Hence, we have
\begin{eqnarray} \label{infdiam}
\frac1{C}\Lambda^{-\liminf_{n\ra\infty} (X_n,Y_n)}&\leq& \lim_{n\ra\infty}\diam(X_n\cup Y_n)\\
 &&\hspace{2cm}\leq C \Lambda^{-\liminf_{n\ra\infty} (X_n,Y_n)}. \nonumber
\end{eqnarray}
Since
\[d(x,y)=\lim_{n\ra\infty}\diam(X_n,Y_n)\] and
\[(\xi,\xi')\leq \liminf_{n\ra\infty} (X_n,Y_n),\]
by inequality \eqref{infdiam}
\begin{eqnarray} \label{dleq}
d(x,y)&=&\lim_{n\ra\infty}\diam(X_n\cup Y_n) \nonumber\\
 &\leq& {C} \Lambda^{-\liminf_{n\ra\infty} (X_n,Y_n)}\\
 &\leq& {C} \Lambda^{- (\xi,\xi')}.\nonumber
\end{eqnarray}
Since \[(\xi,\xi')= \inf \liminf_{n\ra\infty} (X_n,Y_n)\] where infimum is taken for all $\{X_n\}\in \xi$ and $\{Y_n\}\in \xi'$, by inequality \eqref{infdiam},
\begin{eqnarray*}
\frac1{C}\sup\Lambda^{-\liminf_{n\ra\infty} (X_n,Y_n)}\leq \lim_{n\ra\infty}\diam(X_n\cup Y_n) =d(x,y)
\end{eqnarray*}
where supremum is taken for all $\{X_n\}\in \xi$ and $\{Y_n\}\in \xi'$. Hence,
\begin{eqnarray} \label{dgeq}
\frac1{C}\Lambda^{-(\xi,\xi')}=\frac1{C}\Lambda^{-\inf\liminf_{n\ra\infty} (X_n,Y_n)}\leq d(x,y).
\end{eqnarray}
Combining equations \eqref{dleq} and \eqref{dgeq}, we get that
\begin{eqnarray} \label{dleqgeq}
\frac1{C}\Lambda^{-(\xi,\xi')}\leq d(x,y)\leq 4C\Lambda^{-(\xi,\xi')},
\end{eqnarray}
so
\begin{eqnarray*}
\frac1{C}\Lambda^{-(\xi,\xi')}\leq d(h(\xi),h(\xi'))\leq C\Lambda^{-(\xi,\xi')}
\end{eqnarray*}
for all $\xi,\xi'\in \partial_{\infty}\G$.
Therefore, the pull-back of the metric $d$ on $\S^2$ under $h$ is a visual metric on $\partial_{\infty}\G$.

Since $d$ is a visual metric with respect to $f$, equation \eqref{dleqgeq} implies that there exists a constant $c\geq 0$ such that for all $x,y\in \S^2$, and $\xi=h^{-1}(x)$, $\xi=h^{-1}(y)$,
\begin{eqnarray} \label{mxi}
(\xi,\xi')-c\leq m(x,y)\leq (\xi,\xi')+c.
\end{eqnarray}
Let $\rho$ be a visual metric on $\partial_{\infty}\G$ on the Gromov hyperbolic space, so there exists constant $k\geq 1$, such that for any $\xi,\xi'\in \partial_{\infty}\G$,
\[ \frac1{k}\Lambda^{-(\xi,\xi')}\leq \rho(\xi,\xi' ) \leq k\Lambda^{-(\xi,\xi')}.\]
By equation \eqref{mxi}, there exists a constant $k'\geq 1$, such that
\[\frac1{k'}\Lambda^{-(x,y)}\leq \frac1{k}\Lambda^{-(\xi,\xi')}\leq \rho(h^{-1}(x),h^{-1}(y) ) \leq k\Lambda^{-(\xi,\xi')}\leq k'\Lambda^{-m(x,y)},\]
where $x,y\in \S^2$, $\xi=h^{-1}(x)$ and $\xi=h^{-1}(y)$.
Therefore, the pull-back of the metric $\rho$ on $\partial_{\infty}\G$ under $h^{-1}$ is a visual metric on $\S^2$.
\end{proof}

For any Jordan curves $\mathcal C$ and $\mathcal C'$ containing $\post (f)$, let $\partial_{\infty}\G=\partial_{\infty}\G(f,\mathcal C)$ and $\partial_{\infty}\G'=\partial_{\infty}\G(f,\mathcal C')$ be the boundary at infinity of the tile graphs $\G(f,\mathcal C)$ and $\G(f,\mathcal C')$ respectively. By the proposition above, there exist identifications
\[h\: \partial_{\infty}\G \ra \S^2\]
and
\[h'\: \partial_{\infty}\G' \ra \S^2.\]
So we have the following diagram
\[
\xymatrix@R=0.5cm{
  \partial_{\infty}\G \ar[dd]_{g_{\infty}} \ar[dr]^{h}             \\
                & \S^2\ar[dl]^{(h')^{-1}}         \\
  \partial_{\infty}\G'                }
\]
This induced bijection $g_{\infty}=(h')^{-1}\circ h$ should be the same as $g_{\infty}$ as in the remark after Proposition \ref{roughgg}. In addition, under this identification,  visual metrics on $\partial_{\infty}\G$ and $\partial_{\infty}\G'$  are also identified. This is the following corollary.

\begin{cor} \label{corsamevisual}
For any Jordan curves $\mathcal C$ and $\mathcal C'$ containing $\post (f)$, there exists a natural identification between $\G=\partial_{\infty}\G(f,\mathcal C)$ and $\G'=\partial_{\infty}\G(f,\mathcal C')$. Under this identification, a metric $\rho$ is a visual metric on $\partial_{\infty}\G$ if and only if it is a visual metric on $\partial_{\infty}\G'$.
\end{cor}


\section{Asymptotic Upper Curvature}
\noindent
In this section, we define the asymptotic upper curvature for an expanding Thurston map. After review the definition of Latt\`es maps, we give a curvature characterization of Latt\`es maps.

Let $f \: \S^2 \ra S^2$ be an expanding Thurston map.
We define the \emph{asymptotic upper curvature} of $f$ as
\begin{eqnarray} \label{asyf}
K_u(f)\:=K_u(\G(f,\mathcal C)),
\end{eqnarray}
where $\mathcal C \subset \S^2$ is any Jordan curve containing $\post(f)$ and $\G=\G(f,\mathcal C)$ denotes
the Gromov hyperbolic graph constructed from the cell decompositions of $(f, \mathcal C)$.
For any Jordan curves $\mathcal C$ and $\mathcal C'$, the Gromov hyperbolic space $\G(f,\mathcal C)$ and $\G(f,\mathcal C')$ are rough-isomeric by Proposition \ref{roughgg}, and the asymptotic upper curvature is invariant under rough-isometry, so
\[K_u(\G(f,\mathcal C))=K_u(\G(f,\mathcal C')). \]
Therefore, the asymptotic upper curvature $K(f)$ is well-defined in equation \eqref{asyf}.

A \emph{Latt\`es map} $f\:\widehat\C\ra \widehat\C$ is a rational map that is obtained from a finite quotient of a conformal torus endomorphism, i.e., the map $f$ satisfies the following commutative diagram:
\begin{equation}\label{lat}
    \begin{CD}
\T @>\bar{A}>> \T\\
@V\Theta VV @VV\Theta V\\
\widehat\C @>f>> \widehat\C
\end{CD}
\end{equation}
where $\bar A$ is a map of a torus $\T$ that is a quotient of an affine map of the complex plane, and $\Theta$ is a finite-to-one holomorphic map. Latt\`es maps were the first examples of rational maps whose Julia set is the whole sphere $\widehat \C$, and a Latt\`es map is an expanding Thurston map. In \cite{YinLattes}, we have the following combinatorial characterization of Latt\`es maps:
\begin{thm}[Yin, 2011] \label{main0}
A map $f\:\S^2\ra \S^2$ is topologically conjugate to a Latt\`es map if and only if the following conditions hold:
\begin{itemize}
  \item $f$ is an expanding Thurston map;
  \item $f$ has no periodic critical points;
  \item there exists $c>0$ such that $D_n\geq c(\deg f)^{n/2}$ for all $n>0$.
\end{itemize}
\end{thm}
We have the following statement(see Corollary 8.2 in \cite{YinLattes}).
\begin{cor} \label{cormain}
A map $f\:\S^2\ra \S^2$ is topologically conjugate to a Latt\`es map if and only if the followings conditions hold:
\begin{itemize}
  \item $f$ is an expanding Thurston map;
  \item $f$ has no periodic critical points;
  \item there exists a visual metric on $\S^2$ with respect to $f$ with expansion factor $\Lambda=\deg(f)^{1/2}$.
\end{itemize}
\end{cor}

This leads to an curvature characterization of Latt\`es maps as follows.

\begin{thm} \label{main2}
Let $f \: \S^2 \ra \S^2$ be an expanding Thurston map.
The asymptotic upper curvature of $f$ satisfies
\[K_u(f)\geq-\frac14\log^2(\deg f). \]
If in addition, the map $f$ has no periodic critical points, then the tile graph $\G=\G(f)$ is an AC$_u(\kappa)$-space with
\[ \kappa= -\frac14\log^2(\deg f),\]
if and only if the map $f$ is topologically conjugate to a Latt\`es map.
\end{thm}

\begin{proof}
The first part follows directly from the definition of asymptotic upper curvature of $f$ and from Theorem \ref{acku} and Theorem \ref{expansionvisualsup}.

If $f$ is topologically conjugate to a Latt\`es map, then by Corollary \ref{cormain}, there exists a visual metric on $\S^2$ with respect to $f$ with expansion factor $\Lambda=\deg(f)^{1/2}$. By Proposition \ref{samevisual}, there exists a visual metric on $\partial_{\infty}\G$ in the sense of Gromov hyperbolic spaces with expansion factor $\Lambda=\deg(f)^{1/2}$.  By Theorem \ref{acku}, the Gromov hyperbolic space $\G$ is an AC$_u(\kappa)$-space with \[\kappa=-\frac14\log^2(\deg f).\]

Conversely, if $\G$ is an AC$_u(\kappa)$-space with
\[ \kappa= -\frac14\log^2(\deg f),\]
then for all $X,X'\in \G$ and all finite sequences $X_0=X,X_1,\ldots,X_n=X'$ in $\G$,
\begin{eqnarray}\label{acspaceeq1}
(X,X')\geq \min_{i=1,2,\ldots,n}(X_{i-1},X_i)-\frac{\log n}{\log(\deg f)^{1/2}}-c.
\end{eqnarray}

Let $D_n$ be the minimum number of $n$-tiles needed to join opposite sides of Jordan curve $\mathcal C$ as defined in \eqref{defdn}, for $n>0$. For $\#\post(f)\geq 4$, let $P_n=X_1\ldots X_{D_n}$ be an $n$-tile chain joining opposite sides of $\mathcal C$.
By the equation \eqref{acspaceeq1}, we have
\begin{eqnarray*}
(X_1,X_{D_n})\geq \min_{i=1,2,\ldots,D_n}(X_{i-1},X_i)-\frac{\log D_n}{\log(\deg f)^{1/2}}-c,
\end{eqnarray*}
so
\begin{eqnarray} \label{logdndegf}
\frac{\log D_n}{\log(\deg f)^{1/2}}\geq \min_{i=1,2,\ldots,D_n}(X_{i-1},X_i)-(X_1,X_{D_n})-c.
\end{eqnarray}

By equation \eqref{gp}, we have

\begin{equation} \label{XiXi}
\begin{aligned}
(X_{i-1},X_i) &=\frac12[\ell(X_{i-1})+\ell(X_{i})-\eta(X_{i-1},X_i)]+1 \\
&= \frac12[2n-\eta(X_{i-1},X_i)]+1\\
&\geq  n+ \frac12
\end{aligned}
\end{equation}
where $\eta(X_{i-1},X_i)=1$ since $\eta$ is the path metric, and $X_{i-1}$ and $X_i$ have nonempty intersection.
Applying equation \eqref{XiXi} and Lemma \ref{pdiam} to equation \eqref{logdndegf}, we have
\begin{eqnarray*}
\frac{\log D_n}{\log(\deg f)^{1/2}}&\geq &\min_{i=1,2,\ldots,D_n}(X_{i-1},X_i)-(X_1,X_{D_n})-c\\
&\geq& n+ \frac12 
+\log_{\Lambda}\big(\diam(X_1\cup X_{D_n})\big)-\log_{\Lambda}k-c, \\
\end{eqnarray*}
where $k\geq 1$ only depends on $f$ and $\mathcal C$ as in Lemma \ref{pdiam}, and $N>0$ only depends on $f$.
Let $d$ be the minimum length of a line segment joining opposite sides of the Jordan curve, then $d>0$ and
\[\diam(X_1\cup X_{D_n})\geq d. \]
So
\begin{eqnarray*}
\frac{\log D_n}{\log(\deg f)^{1/2}}
&\geq& n+ \frac12 
+\log_{\Lambda}(d)-\log_{\Lambda}k-c \\
&=& n+C.
\end{eqnarray*}
Here the constant
\[C=\frac12+\log_{\Lambda}(d)-\log_{\Lambda}k-c \]
only depends on $f$, $\mathcal C$ and $\Lambda$.
Hence, we have
\begin{eqnarray*}
{\log D_n}&\geq & (n+C) {\log(\deg f)^{1/2}}.
\end{eqnarray*}
Therefore, we have
\begin{eqnarray} \label{end}
{D_n}&\geq &   C'(\deg f)^{n/2},
\end{eqnarray}
where $C'=(\deg f)^{C/2}$  only depends on $f$ and $\mathcal C$.

When  $\#\post(f)=3$, the argument above holds with minor modifications. More specifically, consider a set $S$ of $D_n$ $n$-tiles joining opposite sides of $\mathcal C$. By a simple compactness argument, there exists a positive lower bound $d$ on the diameter of any set intersecting all three edges of $\mathcal C$.
Hence, we may choose $n$-tiles $X$ and $Y$ in $S$ that intersect two distinct $0$-edges and satisfy $\diam(X \cup Y) \geq d$. Then one can choose $P_n=X_1\ldots X_{r_n}$ to be an $n$-tile chain joining opposite sides of $\mathcal C$, such that $X_1 = X$, $X_{r_n} = Y$, $X_i \in S$ for $i =1, \ldots, r_n$, and $r_n<2D_n$.


By Theorem \ref{main0}, the map $f$ is topologically conjugate to a Latt\`es map.
\end{proof}

\newpage

\bibliographystyle{amsplain}
\bibliography{qian}
\end{document}